\documentclass[12pt,reqno]{amsart}
\usepackage{amsthm}
\usepackage{amssymb}
\usepackage{graphics}
\usepackage{tikz}
\usetikzlibrary{shapes,backgrounds,calc}
\usepackage{latexsym}
\usepackage{multicol}
\usepackage{verbatim,enumerate}
\usepackage{accents}
\usepackage{cite}
\usepackage{multirow}
\usepackage{bigstrut}
\usepackage{amsthm}
\usepackage{amssymb}
\usepackage{graphics}
\usepackage{tikz}
\usetikzlibrary{shapes,backgrounds,calc}
\usepackage{latexsym}
\usepackage{multicol}
\usepackage{verbatim,enumerate}
\usepackage{accents}
\usepackage{cite}
\usepackage{array}
\usepackage[colorlinks=true, linkcolor=blue, citecolor=blue, urlcolor=black]{hyperref}
\usepackage{hyperref}
\usepackage{amsmath, amscd,url}
\usepackage{multirow}
\usepackage{bigstrut}
\usepackage{longtable}

\usepackage{stackengine}

\usepackage{hyperref}
\usepackage{amsmath, amscd,url}
\usepackage{longtable}

\advance\textwidth by 1.3in \advance\oddsidemargin by -.6in \advance\evensidemargin by -.6in
\theoremstyle{definition}

\newtheorem{theorem}{Theorem}[section]
\newtheorem{lemma}[theorem]{Lemma}

\newtheorem{corollary}[theorem]{Corollary}
\theoremstyle{definition}
\newtheorem{definition}[theorem]{Definition}
\newtheorem{example}[theorem]{Example}

\theoremstyle{remark}
\newtheorem{remark}[theorem]{Remark}

\theoremstyle{definition}

\newcounter{cnt}
 \makeatletter
\def\mydggeometry{\makeatletter\dg@YGRID=1\dg@XGRID=20\unitlength=0.003pt\makeatother}
\makeatother \theoremstyle{remark}

\numberwithin{equation}{section}
\let\bwdg\bigwedge
\def\bigwedge{{\textstyle\bwdg}}

\newcommand{\nc}{\newcommand}
\newcommand{\rnc}{\renewcommand}

\nc{\cal}{\mathcal} \nc{\goth}{\mathfrak} \rnc{\bold}{\mathbf}

\nc\bomega{{\mbox{\boldmath $\omega$}}} \nc\bpsi{{\mbox{\boldmath $\Psi$}}}
 \nc\balpha{{\mbox{\boldmath $\alpha$}}}
 \nc\bpi{{\mbox{\boldmath $\pi$}}}
 \nc\bvpi{{\mbox{\boldmath $\varpi$}}}
\nc\chara{\operatorname{ch}}

  \nc\bxi{{\mbox{\boldmath $\xi$}}}
\nc\bmu{{\mbox{\boldmath $\mu$}}} \nc\bcN{{\mbox{\boldmath $\cal{N}$}}} \nc\bcm{{\mbox{\boldmath $\cal{M}$}}} \nc\blambda{{\mbox{\boldmath
$\lambda$}}}\nc\bnu{{\mbox{\boldmath $\nu$}}}

\makeatletter
\def\section{\def\@secnumfont{\mdseries}\@startsection{section}{1}%
  \z@{.7\linespacing\@plus\linespacing}{.5\linespacing}%
  {\normalfont\scshape\centering}}
\def\subsection{\def\@secnumfont{\bfseries}\@startsection{subsection}{2}%
  {\parindent}{.5\linespacing\@plus.7\linespacing}{-.5em}%
  {\normalfont\bfseries}}
\makeatother

 \nc{\Hom}{\operatorname{Hom}}
  \nc{\mode}{\operatorname{mod}}
\nc{\End}{\operatorname{End}} \nc{\wh}[1]{\widehat{#1}} \nc{\Ext}{\operatorname{Ext}} \nc{\ch}{\text{ch}} \nc{\ev}{\operatorname{ev}}
\nc{\Ob}{\operatorname{Ob}} \nc{\soc}{\operatorname{soc}} \nc{\rad}{\operatorname{rad}} \nc{\head}{\operatorname{head}}

 \nc{\Cal}{\cal} \nc{\Xp}[1]{X^+(#1)} \nc{\Xm}[1]{X^-(#1)}
\nc{\on}{\operatorname} \nc{\Z}{{\bold Z}} \nc{\J}{{\cal J}}  \nc{\Q}{{\bold Q}}

\nc{\N}{{\bold N}}  \nc\boa{\bold a} \nc\bob{\bold b} \nc\boc{\bold c} \nc\bod{\bold d} \nc\boe{\bold e} \nc\bof{\bold f} \nc\bog{\bold g}
\nc\boh{\bold h} \nc\boi{\bold i} \nc\boj{\bold j} \nc\bok{\bold k} \nc\bol{\bold l} \nc\bom{\bold m} \nc\bon{\mathbb n} \nc\boo{\bold o}
\nc\bop{\bold p} \nc\boq{\bold q} \nc\bor{\bold r} \nc\bos{\bold s} \nc\boT{\bold t} \nc\boF{\bold F} \nc\bou{\bold u} \nc\bov{\bold v}
\nc\bow{\bold w} \nc\boz{\bold z}\nc\ba{\bold A} \nc\bb{\bold B} \nc\bc{\mathbb C} \nc\bd{\bold D} \nc\be{\bold E} \nc\bg{\bold
G} \nc\bh{\bold H} \nc\bi{\bold I} \nc\bj{\bold J} \nc\bk{\bold K} \nc\bl{\bold L} \nc\bm{\bold M} \nc\bn{\mathbb N} \nc\bo{\bold O} \nc\bp{\bold
P} \nc\bq{\bold Q} \nc\br{\bold R} \nc\bs{\bold S} \nc\bt{\bold T} \nc\bu{\bold U} \nc\bv{\bold V} \nc\bw{\bold W} \nc\bz{\mathbb Z} \nc\bx{\bold
x} \nc\KR{\bold{KR}} \nc\rk{\bold{rk}} \nc\het{\text{ht }}

\nc\toa{\tilde a} \nc\tob{\tilde b} \nc\toc{\tilde c} \nc\tod{\tilde d} \nc\toe{\tilde e} \nc\tof{\tilde f} \nc\tog{\tilde g} \nc\toh{\tilde h}
\nc\toi{\tilde i} \nc\toj{\tilde j} \nc\tok{\tilde k} \nc\tol{\tilde l} \nc\tom{\tilde m} \nc\ton{\tilde n} \nc\too{\tilde o} \nc\toq{\tilde q}
\nc\tor{\tilde r} \nc\tos{\tilde s} \nc\toT{\tilde t} \nc\tou{\tilde u} \nc\tov{\tilde v} \nc\tow{\tilde w} \nc\toz{\tilde z} \nc\woi{w_{\omega_i}}

\begin{document}
\setcounter{section}{0}
\setcounter{tocdepth}{1}

%%%%%%%%%%%%%%%%%%%%%%%%%%%%%%%%%%%%%%%%%%%%

\title{On Schur's irreducibility results and generalised $\phi$-Hermite polynomials}

\author[Anuj Jakhar]{Anuj Jakhar}
%\author[Srinivas Kotyada]{Srinivas Kotyada}
\address[Anuj Jakhar]{Department of Mathematics, Indian Institute of Technology (IIT) Madras}
%\address[Srinivas Kotyada]{Department of Mathematics, The Institute of Mathematical Sciences (IMSc) Chennai}
%\address[Surender Kumar]{Department of Mathematics, Indian Institute of Technology (IIT) Bhilai}

\email[Anuj Jakhar]{anujjakhar@iitm.ac.in \\ anujiisermohali@gmail.com}
%\email[Srinivas Kotyada]{srini@imsc.res.in}
%\email[Surender Kumar]{surenderk@iitbhilai.ac.in}

%\thanks{The first author is employed at IIT Madras and thankful to SERB grant SRG/2021/000393.}


\begin{abstract}
\noindent Let $c$ be a fixed integer such that $c \in \{0,2\}.$ Let $n$ be a positive integer such that either $n\geq 2$ or $2n+1 \neq 3^u$ for any integer $u\geq 2$ according as $c = 0$ or not. Let $\phi(x)$ belonging to $\Z[x]$ be a monic polynomial which is irreducible modulo all primes less than $2n+c$. Let $a_i(x)$ with $0\leq i\leq n-1$ belonging to $\Z[x]$  be polynomials having degree less than $\deg\phi(x)$. Let $a_n \in \Z$ and the content of $(a_na_0(x))$ is not divisible by any prime less than $2n+c$.  For a positive integer $j$, if $u_j$ denotes the product of the odd numbers $\leq j$, then we show that the polynomial $\frac{a_{n}}{u_{2n+c}}\phi(x)^{2n}+\sum\limits_{j=0}^{n-1}a_j(x)\frac{\phi(x)^{2j}}{u_{2j+c}}$ is irreducible over the field $\Q$ of rational numbers. This generalises a well-known result of Schur which states that the polynomial   $\sum\limits_{j=0}^{n}a_j\frac{x^{2j}}{u_{2j+c}}$ with $a_j \in \Z$ and $|a_0| = |a_n| = 1$ is irreducible over $\Q$. We illustrate our result through examples.
\end{abstract}
\maketitle

\section{Introduction and statements of results}\label{intro}
For each non-negative integer $j$, we define $u_j$ as the product of the odd numbers $\leq j$. In particular, we have $u_0 = u_2 = 1$, $u_4 = 3,~u_6=15,~\cdots.$ 

In 1929, Schur\cite{Sch} proved the following two results.
\begin{theorem}\label{a1}
	Let $a_j$'s for $0\leq j\leq n$, be integers and $|a_0| = |a_n| = 1$. Then the polynomial   $\sum\limits_{j=0}^{n}a_j\frac{x^{2j}}{u_{2j}}$  is irreducible over $\Q$.
\end{theorem}
\begin{theorem}\label{a2}
	Let $n$ be a positive integer such that $2n+1 \neq 3^u$ for any integer $u\geq 2$. Let $a_j$'s for $0\leq j\leq n$, be integers and $|a_0| = |a_n| = 1$. Then the polynomial   $\sum\limits_{j=0}^{n}a_j\frac{x^{2j}}{u_{2j+2}}$  is irreducible over $\Q$.
\end{theorem}
As a consequence of the above theorems, he also proved that the $m^{th}$ classical Hermite polynomial, given by
$$H_m(x) = \sum\limits_{j=0}^{[m/2]}(-1)^j\binom{m}{2j}u_{2j}x^{m-2j},$$
 is irreducible if $m$ is even and is $x$ times an irreducible polynomial if $m$ is odd.

In the present paper, we extend Theorems \ref{a1}, \ref{a2}. More precisely, we prove the following results.
\begin{theorem}\label{1.1}
	Let $n \geq 2$ and $a_n$ be integers. Let $\phi(x)$ belonging to $\Z[x]$ be a monic polynomial which is irreducible modulo all primes less than $2n$. Suppose that $a_0(x), \cdots, a_{n-1}(x)$ belonging to $\Z[x]$ satisfy the following conditions.
	\begin{itemize}
\item[(i)] $\deg a_i(x) < \deg \phi(x)$ for $0\leq i\leq n-1$,
\item[(ii)] the content of $(a_na_0(x))$ is not divisible by any prime less than $2n$.
%\item[(iii)] $a_n = \pm 1.$
	\end{itemize}
	Then the polynomial $$f_1(x)= \frac{a_{n}}{u_{2n}}\phi(x)^{2n}+\sum\limits_{j=0}^{n-1}a_j(x)\frac{\phi(x)^{2j}}{u_{2j}}$$  is irreducible over $\Q$.
\end{theorem}

We wish to point out here that the analogues of the above theorem does not hold for $n=1$ because if $\phi(x) \in \Z[x]$ is a monic polynomial of degree $m \geq 3$, then the polynomial $\phi(x)^2 - x^2 = (\phi(x) + x)(\phi(x) - x)$ is reducible over $\Q$.  

It may be pointed out that in Theorem \ref{1.1} the assumption ``the content of $a_0(x)$ is not divisible by any prime less than $2n$" cannot be dispensed with. For example, consider the polynomial $\phi(x)=x^2-x+5$ which is irreducible modulo $2$ and $3$. Then the polynomial $f(x) = \frac{\phi(x)^4}{3}-3$ = $\frac{1}{3}(\phi(x)^2+3)(\phi(x)^2-3)$ is reducible over $\Q$.  

We also give below an example to show that Theorem \ref{1.1} may not hold if $a_n$ is replaced by a (monic) polynomial $a_n(x)$ with integer coefficient having degree less than $\deg\phi(x)$. Consider $\phi(x) = x^2-x+5$ which is irreducible modulo $2$ and $3$. Take $a_2(x) = x-3, a_1(x) = x+26$ and $a_0(x) = 5(x-5)$. Then the polynomial $$a_2(x)\frac{\phi(x)^4}{3} + a_1(x)\phi(x)^2 + a_0(x)$$ has $0$ as a root.

\begin{theorem}\label{1.1'}
	Let $n$ be a positive integer such that $2n+1 \neq 3^u$ for any integer $u\geq 2$.  Let $\phi(x)$ belonging to $\Z[x]$ be a monic polynomial which is irreducible modulo all primes less than $2n+2$. Suppose that $a_n, a_0(x), \cdots, a_{n-1}(x)$ belonging to $\Z[x]$ satisfy the following conditions.
	\begin{itemize}
\item[(i)] $\deg a_i(x) < \deg \phi(x)$ for $0\leq i\leq n-1$,
\item[(ii)] the content of $(a_na_0(x))$ is not divisible by any prime less than $2n+2$.
%\item[(iii)] $a_n = \pm 1.$
	\end{itemize}
	Then the polynomial $$f_2(x)= \frac{a_{n}}{u_{2n+2}}\phi(x)^{2n}+\sum\limits_{j=0}^{n-1}a_j(x)\frac{\phi(x)^{2j}}{u_{2j+2}}$$  is irreducible over $\Q$.
\end{theorem}

It may be pointed out that in Theorem \ref{1.1'} the assumption ``the content of $a_0(x)$ is not divisible by any prime less than $2n+2$" cannot be dispensed with. For example, consider the polynomial $\phi(x)=x^2-x+11$ which is irreducible modulo $2, 3$ and $5$. Then the polynomial $f(x) = \frac{\phi(x)^4}{15} + 6\frac{\phi(x)^2}{3} +15$ = $\frac{1}{15}(\phi(x)^2+15)^2$ is reducible over $\Q$.  

We also give below an example to show that Theorem $\ref{1.1'}$ may not hold if $a_n$ is replaced by a (monic) polynomial $a_n(x)$ with integer coefficient having degree less than $\deg\phi(x)$. Consider $\phi(x) = x^2-x+11$ which is irreducible modulo $2, 3$ and $5$. Take $a_2(x) = x-15, a_1(x) = x+366$ and $a_0(x) = x-121$. Then the polynomial $$a_2(x)\frac{\phi(x)^4}{15} + a_1(x)\frac{\phi(x)^2}{3} + a_0(x)$$ has $0$ as a root.

As an application of Theorems \ref{1.1}, \ref{1.1'}, we shall prove the following theorem. 

\begin{theorem}\label{Her}
	Let $m\geq 3$ be an integer. Let $\phi(x)$ belonging to $\Z[x]$ be a monic polynomial which is irreducible modulo all primes $\leq m$. Suppose that $a_{[\frac{m}{2}]}, a_0(x), \cdots, a_{[\frac{m}{2}]-1}(x)$ belonging to $\Z[x]$ satisfy the following conditions.
	\begin{itemize}
\item[(i)] $\deg a_i(x) < \deg \phi(x)$ for $0\leq i\leq [m/2]-1$,
\item[(ii)] the content of $(a_{[\frac{m}{2}]}a_0(x))$ is not divisible by any prime $\leq m$.
%\item[(iii)] $a_n = \pm 1.$
	\end{itemize}
	Then the polynomial $$H^{\phi}_m(x)= a_{[\frac{m}{2}]}\phi(x)^m + \sum\limits_{j=1}^{[m/2]}\binom{m}{2j}u_{2j}a_{[\frac{m}{2}]-j}(x)\phi(x)^{m-2j}$$  is irreducible over $\Q$ if $m$ is even and is $\phi(x)$ times an irreducible polynomial if $m$ is odd with $m\neq 3^u$ for any integer $u\geq 2$.
\end{theorem}

We wish to point out here that in the above theorem, if we take $\phi(x) = x$ and $a_{[\frac{m}{2}]} =(-1)^{[\frac{m}{2}]}, ~ a_i(x) = (-1)^i$ for $0\leq i\leq [m/2]-1$, then  $H^{\phi}_m(x)$ becomes $m$-th classical Hermite polynomial.

The following corollary is an immediate consequence of the above theorem.
\begin{corollary}\label{Her1}
	Let $m\geq 3$ be an integer. Let $\phi(x)$ belonging to $\Z[x]$ be a monic polynomial which is irreducible modulo all primes less than $m$. 	Then the polynomial $$H_m(\phi(x))= \sum\limits_{j=0}^{[m/2]}(-1)^j\binom{m}{2j}u_{2j}\phi(x)^{m-2j}$$  is irreducible over $\Q$ if $m$ is even and is $\phi(x)$ times an irreducible polynomial if $m$ is odd with $m\neq 3^u$ for any integer $u\geq 2$.
\end{corollary}

We now provide some examples of Theorems \ref{1.1}, \ref{1.1'}. 

\begin{example}
	Consider $\phi(x) = x^3-x+37$. It can be easily checked that $\phi(x)$ is irreducible modulo $2,3,5$ and $7$. Let $j\geq 2$ and $a_j$ be integers. Let $a_{i}(x) \in \Z[x]$ be polynomials each having degree less than $3$ for $0\leq i\leq j-1$. Assume that the content of $(a_ja_0(x))$ is not divisible by any prime less than $2j$. Then by Theorem \ref{1.1}, the polynomial	$$f_j(x) = \frac{a_{j}}{u_{2j}}\phi(x)^{2j}+\sum\limits_{i=0}^{j-1}a_i(x)\frac{\phi(x)^{2i}}{u_{2i}}$$  is irreducible over $\Q$ for $j \in \{2,3,4,5\}$.
\end{example}
\begin{example}
	Consider $\phi(x) = x^2-x+17$. It can be easily checked that $\phi(x)$ is irreducible modulo $2,3,5$ and $7$. Let $j\geq 2$ and $a_j$ be integers. Let $a_{i}(x) \in \Z[x]$ be polynomials each having degree less than $2$ for $0\leq i\leq j-1$. Assume that the content of $(a_ja_0(x))$ is not divisible by any prime less than $2j+2$. Then by Theorem \ref{1.1'}, the polynomial	$$g_j(x) = \frac{a_{j}}{u_{2j+2}}\phi(x)^{2j}+\sum\limits_{i=0}^{j-1}a_i(x)\frac{\phi(x)^{2i}}{u_{2i+2}}$$  is irreducible over $\Q$ for $j \in \{2,3,4\}$.
\end{example}
It may be pointed out that, in the above examples, the irreducibility of the polynomials $f_5(x)$ and $g_4(x)$ do not seem to follow from any known irreducibility criterion (cf. \cite{brown2008}, \cite{JakBLMS}, \cite{JakAMS}, \cite{JakJA}, \cite{JakAM}, \cite{Ja-Sa} and \cite{Jho-Kh}).

\section{Preliminary results.}
We first introduce the notion of Gauss valuation and $\phi$-Newton polygon. For a prime $p$, $v_p$ will denote the $p$-adic valuation of $\Q$ defined for any non-zero integer $b$ to be the highest power of $p$ dividing $b$. 
We shall denote by $v_p^x$ the Gaussian valuation extending $v_p$ defined on the polynomial ring $\Z[x]$ by $$v_p^x(\sum\limits_{i}b_ix^i) = \min_{i}\{v_p(b_i)\}, b_i \in \Z.$$ %We now define the notion of $\phi$-Newton polygon.
\begin{definition}\label{def1}
	Let $p$ be a prime number and $\phi(x) \in \Z[x]$ be a monic polynomial which is irreducible modulo $p$. Let $f(x)$ belonging to $\Z[x]$ be a polynomial having $\phi$-expansion\footnote{If $\phi(x)$ is a fixed monic polynomial with coefficients in $\Z$, then any $f(x)\in \Z[x]$ can be uniquely written as a finite sum $\sum\limits_{i}b_i(x)\phi(x)^i$ with $\deg b_i(x)< \deg \phi(x)$ for each $i$; this expansion will be referred to as the $\phi$-expansion of $f(x)$.} $\sum\limits_{i=0}^{n}b_i(x) \phi(x)^i$ with $b_0(x)b_n(x) \neq 0$. Let $P_i$ stand for the point in the plane having coordinates $(i, v_p^x(b_{n-i}(x)))$ when $b_{n-i}(x)\neq 0$, $0\leq i \leq n$.
Let $\mu_{ij}$ denote the slope of the line joining the points $P_i$ and $P_j$ if $b_{n-i}(x)b_{n-j}(x)\ne 0$. Let $i_1$ be the largest index $0< i_1 \leq n$ such that 
\begin{align*}
	\mu_{0i_1}=\min \{\mu_{0j}\ |\ 0<j \leq n,\ b_{n-j}(x)\ne 0\}.
\end{align*}
If $i_1<n$, let $i_2$ be the largest index $i_1< i_2 \leq n$ such that 
\begin{align*}
	\mu_{i_1i_2}=\min \{\mu_{i_1j}\ |\ i_1<j \leq n,\ b_{n-j}(x)\ne 0\}
\end{align*}
and so on. The $\phi$-Newton polygon of $f(x)$ with respect to $p$ is the polygonal path having segments $P_0P_{i_1}, P_{i_1}P_{i_2}, \dots, P_{i_{k-1}}P_{i_k}$ with $i_k=n$. These segments are called the edges of the $\phi$-Newton polygon of $f(x)$ and their slopes from left to right form a strictly increasing sequence. The $\phi$-Newton polygon minus the horizontal part (if any) is called its principal part.
\end{definition}

We shall use the following lemma\cite[Theorem 1.3]{Ji-Kh} in the sequel. We omit its proof.
%With the above notation, we prove the following theorem, which was proved by Filaseta \cite[Lemma 2]{filaseta1995} in 1995 in the particular case when $\phi(x)=x$; the proof given here is similar to the one given by him.

\begin{lemma} \label{fila}
	Let $n, k$ and $\ell$ be integers with $0\leq \ell<k\leq \frac n2$ and $p$ be a prime. Let $\phi(x)\in \Z[x]$ be a monic polynomial which is irreducible modulo $p$. Let $f(x)$ belonging to $\Z[x]$ be a monic polynomial not divisible by $\phi(x)$ having $\phi$-expansion $\sum\limits_{i=0}^{n}f_i(x) \phi(x)^i$ with $f_n(x)\ne 0$. Assume that $v_p^x(f_i(x))> 0$ for $0\leq i\leq n-\ell-1$ and the right-most edge of the $\phi$-Newton polygon of $f(x)$ with respect to $p$ has slope less than $\frac 1k$. Let $a_0(x), a_1(x), \dots,a_n(x)$ be polynomials over $\Z$ satisfying the following conditions. 
	\begin{itemize}
		\item[(i)] $\deg a_i(x)< \deg \phi(x)-\deg f_i(x)$ for $0\leq i \leq n$, 
		\item[(ii)] $v_p^x(a_0(x))=0$, i.e., the content of $a_0(x)$ is not divisible by $p$,
		\item[(iii)] the leading coefficient of $a_n(x)$ is not divisible by $p$.
	\end{itemize}
	Then the polynomial $\sum\limits_{i=0}^{n} a_i(x) f_i(x) \phi(x)^i$	does not have a factor in $\Z[x]$ with degree lying in the interval $[(\ell+1)\deg \phi(x), (k+1)\deg \phi(x)).$
\end{lemma}

	We now prove the following elementary result will be used in the proof of Theorems \ref{1.1}, \ref{1.1'}.

\begin{lemma}\label{1.3}
	Let $n\geq 1$ be an integer and $p$ be a prime number. Let $\phi(x)\in \Z[x]$ be a monic polynomial which is irreducible modulo $p$. Suppose that $a_0(x), a_1(x), \cdots, a_{n-1}(x)$ belonging to $\Z[x]$ are polynomials each having degree less than $\deg \phi(x)$. Let $a_n$ be an integer with $p\nmid a_n$. Let $b_0, b_1, \cdots, b_{n-1}$ be integers such that $p|b_j$ for each $j$, $0\leq j\leq n-1$. Then the polynomial $$F(x)= a_{n}\phi(x)^{2n}+\sum\limits_{j=0}^{n-1}b_ja_j(x)\phi(x)^{2j}$$ can not have any non-constant factor having degree less than $\deg \phi(x)$.
\end{lemma}
\begin{proof}
	%It is enough to show that $F(x) = a_n\phi(x)^{2n} + \sum\limits_{j=0}^{n-1}a_j(x)\frac{u_{2n}\phi(x)^{2j}}{u_{2j}}$ does not have a non-constant factor over $\Z$ with degree less than $\deg\phi(x)$. 
	%Since the leading coefficient of $f(x)$ is $a_n$, it is not divisible by $p$. 
	Let $c$ denote the content of $F(x)$. As $p\nmid a_n$, we have $p\nmid c$. Now suppose to the contrary that there exists a primitive non-constant polynomial $h(x) \in \Z[x]$ dividing $F(x)$ having degree less than $\deg \phi(x)$. Then in view of Gauss Lemma, there exists $g(x)\in \Z[x]$ such that $\frac{F(x)}{c} = h(x)g(x)$. The leading coefficient of $F(x)$ and hence those of $h(x)$ and $g(x)$ are coprime with $p$. Note that $p$ divides $b_j$ for $0\leq j\leq n-1$. Therefore on passing to $\Z/p\Z$, we see that the degree of $\bar{h}(x)$ is same as that of $h(x)$. Hence $\deg\bar{h}(x)$ is positive and less than $\deg\phi(x)$. This is impossible because $\bar{h}(x)$ is a divisor of $\frac{\overline{F}(x)}{\bar{c}} = \frac{\bar{a}_n}{\bar{c}}\bar{\phi}(x)^{2n}$ and $\bar{\phi}(x)$ is irreducible over $\Z/p\Z$. This completes the proof of the lemma.
	\end{proof}

The following result is due to Schur \cite{Sch} and plays an important role in the proof of Theorems \ref{1.1}, \ref{1.1'}.
\begin{lemma}\label{prime}
	For integers $k$ and $n$ with $n>k>2$, at least one of the $k$ numbers $2n+1, 2n+3, \cdots, 2n+2k-1$ is divisible by a prime $p > 2k+1$. For $k = 2$, the same result holds unless $2n+1 = 25$. For $k = 1$, the same result holds unless $2n+1 = 3^u$ for some integer $u \geq 2.$
\end{lemma}
\begin{remark}\label{prime'}
	It can be noted that the first part of the above lemma can be rephrased as saying that, for $k > 2$, the product of any $k$ consecutive odd numbers each $> 2k+1$ is divisible by a prime $p>2k+1$.
\end{remark}
\section{Proof of Theorem \ref{1.1}}
%Before proving Theorem \ref{1.1}, we first prove the following theorem which is of independent interest as well. %This theorem extends the result proved by Allen and Filaseta (cf. \cite[Lemma 6]{AF}). 
%It will be used in the proof of Theorem \ref{1.1}.
%\begin{theorem}\label{1.2}
%	Let $\phi(x)\in \Z[x]$ be a monic polynomial which is irreducible modulo all primes less than $2n$ and $a_0(x), a_1(x), \cdots, a_n(x)$ belonging to $\Z[x]$ be polynomials each having degree less than $\deg \phi(x)$. Let the content of $a_0(x)$ is not divisible by any prime less than $2n$ and $$f(x) = \sum\limits_{j=0}^{n}a_j(x)\frac{\phi(x)^{2j}}{u_{2j}}.$$ Let $k$ be a positive odd integer $\leq n$. Suppose that there exists a prime $p \geq k+2$ and a positive integer $r$ for which $$p^r | (2n-1)(2n-3)\cdots (2n-k)$$ and the leading coefficient of $a_n(x)$ is not divisible by $p$. Then $f(x)$ can not have a factor in $\Z[x]$ with degree lying in the interval $[k\deg\phi(x), (k+2)\deg\phi(x)).$
%\end{theorem}

\begin{proof}[Proof of Theorem \ref{1.1}]
	By hypothesis, $n\geq 2$. It suffices to show that $F_1(x) = u_{2n}f_1(x)$ can not have a factor in $\Z[x]$ with degree lying in the interval $[1, (n+1)\deg\phi(x))$. If we choose a prime $p$ such that $p$ divides $2n-1$, then we have $p$ divides $\frac{u_{2n}}{u_{2j}}$ for each $0\leq j\leq n-1$. Hence using Lemma \ref{1.3}, we see that $F_1(x)$ can not have any non-constant factor having degree less than $\deg\phi(x)$. Now assume that $F_1(x)$ has a factor  in $\Z[x]$ with degree lying in the interval $[\deg\phi(x), (n+1)\deg\phi(x))$. We make use of Lemma \ref{fila} to obtain a contradiction. 
	We consider a new polynomial $g_1(x)$ with $a_n = 1 = a_j(x)$, $0\leq j\leq n-1$,  in $F_1(x)$ given by $$g_1(x) = \phi(x)^{2n} + (2n-1)\phi(x)^{2n-2} + \cdots + (2n-1)\cdots 5\cdot 3\cdot \phi(x)^2+(2n-1)\cdots 3\cdot 1.$$
	We define integers $c_i$  such that $c_{2n} = 1$, for an odd integer $r$ we have $c_{r} = 0$ and  $c_{2n-(r+1)} = (2n-1)(2n-3)\cdots (2n-r+2)(2n-r)$.  Hence we can write $g_1(x) = \sum\limits_{i=0}^{2n}c_i\phi(x)^{i}$. 
	Observe that for every $i \in \{0,1, \cdots, 2n-1\},$ $c_i$ is divisible by the product of the odd numbers in the interval $(i, 2n-1].$ Also, for $0\leq j\leq n$, $c_{2j} = \frac{u_{2n}}{u_{2j}}$. Let $\ell = k-1$ so that $\ell + 1 = k$. By Lemma \ref{prime}, there is a prime factor $p \geq k+1$ that divides $c_{2n-\ell-1}$ and $c_{2n-\ell-2}$, which implies that $p|c_i$ for all $i \in \{0, 1, \cdots, 2n-\ell -1\}$. Clearly $p\nmid c_{2n}$. The slope of the right-most edge of the $\phi$-Newton polygon of $g_1(x)$ with respect to $p$ can be determined by 
	$$\max\limits_{1\leq j\leq n}\bigg\{\frac{v_p(u_{2n})-v_p(u_{2n}/u_{2j})}{2j}\bigg\}.$$
	Using the fact that $v_p((2j-1)!) < (2j-1)/(p-1)$, for $1\leq j\leq n$ we obtain
$$v_p(u_{2n})-v_p(u_{2n}/u_{2j}) = v_p(u_{2j})\leq v_p((2j-1)!) < \frac{2j-1}{p-1} < \frac{2j}{p-1}.$$
As $p \geq k+1$, we deduce that the slope of the right-most edge of the $\phi$-Newton polygon of $g_1(x)$ with respect to $p$ is $<\frac{1}{k}$. Hence, using Lemma \ref{fila}, we have a contradiction. This completes the proof of the theorem.
	\end{proof}

	\section{Proof of Theorem \ref{1.1'}.}
	\begin{proof}[Proof of Theorem \ref{1.1'}] 
		It suffices to show that $F_2(x) = u_{2n+2}f_2(x)$ can not have a factor in $\Z[x]$ with degree lying in the interval $[1, (n+1)\deg\phi(x))$. If we choose a prime $p$ such that $p$ divides $2n+1$, then we have $p$ divides $\frac{u_{2n+2}}{u_{2j+2}}$ for each $0\leq j\leq n-1$. Hence using Lemma \ref{1.3}, we see that $F_2(x)$ can not have any non-constant factor having degree less than $\deg\phi(x)$. Now assume that $F_2(x)$ has a factor  in $\Z[x]$ with degree lying in the interval $[\deg\phi(x), (n+1)\deg\phi(x))$. We make use of Lemma \ref{fila} to obtain a contradiction. 
	We consider a new polynomial $g_2(x)$ with $a_n = 1 = a_j(x)$, $0\leq j\leq n-1$,  in $F_2(x)$ given by $$g_2(x) = \phi(x)^{2n} + (2n+1)\phi(x)^{2n-2} + \cdots + (2n+1)\cdots 7\cdot 5\cdot \phi(x)^2+(2n+1)\cdots 5\cdot 3.$$
	Now we define integers $c_i$  such that $c_{2n} = 1$, for an odd integer $r$ we have $c_{r} = 0$ and  $c_{2n-(r+1)} = (2n+1)(2n-1)\cdots (2n-r+2)$.  Hence we can write $g_2(x) = \sum\limits_{i=0}^{2n}c_i\phi(x)^{i}$. 
	Observe that for every $i \in \{0,1, \cdots, 2n-1\},$ $c_i$ is divisible by the product of the odd numbers in the interval $(i+2, 2n+1].$ Also, for $0\leq j\leq n$, $c_{2j} = \frac{u_{2n+2}}{u_{2j+2}}$. Let $\ell = k-1$ so that $\ell + 1 = k$. By Lemma \ref{prime}, there is a prime factor $p \geq k+2$ that divides $c_{2n-\ell-1}$ and $c_{2n-\ell-2}$ unless either (i) $k =2$ and $2n+1 = 3^u$ with some integer $u\geq 2$ or (ii) $k = 4$ and $n = 13$. For the moment, suppose we are not in either of the situations described by (i) and (ii), and fix $p \geq k+2$ dividing $c_{2n-\ell-1}$ and $c_{2n-\ell-2}$. Then $p\nmid c_{2n}$ and $p|c_i$ for all $i\in \{0,1,\cdots, 2n-\ell-1\}$. Next, we show that the slope of the right-most edge of the $\phi$-Newton polygon of $g_2(x)$ with respect to $p$ is $<\frac{1}{k}$, but we note here that our argument will only depend on $p$ being a prime $\geq k+2$ and not on $p$ dividing $c_{2n-\ell-1}$ and $c_{2n-\ell-2}$. The slope of the right-most edge of the $\phi$-Newton polygon of $g_2(x)$ with respect to $p$ can be determined by 
	$$\max\limits_{1\leq j\leq n}\bigg\{\frac{v_p(u_{2n+2})-v_p(u_{2n+2}/u_{2j+2})}{2j}\bigg\}.$$
	 For $1\leq j\leq n$, we obtain
$$v_p(u_{2n+2})-v_p(u_{2n+2}/u_{2j+2}) = v_p(u_{2j+2})\leq v_p((2j+1)!).$$
If $p > 2j+1$, then $v_p((2j+1)!) = 0.$ If $p \leq 2j+1,$ then $k+1 \geq 2j$ and, from the fact that $v_p((2j+1)!) < (2j+1)/(p-1)$, we deduce that 
$$v_p((2j+1)!) < \frac{2j+1}{p-1} \leq \frac{2j+1}{k+1} < \frac{2j}{k}.$$
It follows that the slope of the right-most edge of the $\phi$-Newton polygon of $g_2(x)$ with respect to $p$ is $<\frac{1}{k}$. Hence, using Lemma \ref{fila}, we have a contradiction. Since by hypothesis, we have $2n+1 \neq 3^u$ for any integer $u\geq 2$, this contradiction completes the proof of the Theorem \ref{1.1'}.
%So now consider the situations (i) and (ii). 
%(i) $n = 13$ with $k = 4$. In this situation, it remains to prove that $F_2(x)$ should not have a factor in $\Z[x]$ with degree lying in the interval $[4\deg\phi(x), 5\deg\phi(x))$. Note that 
%$$F_2(x) = a_{13}\phi(x)^{26} + 27a_{12}(x)\phi(x)^{24} + \cdots + 27\cdot 25\cdots 7\cdot 5\cdot a_{1}(x)\phi(x)^2 +  27\cdot 25\cdots 5\cdot 3\cdot a_0(x).$$
%Using the hypothesis, $a_n$ and the content of $a_0(x)$ are not divisible by any prime less than $2n+2$. Take $p = 23$. In this case, we see that the $\phi$-Newton polygon of $F_2(x)$ with respect to $23$ has two edges. The first edge joins the point $(0, 0)$ to $(4, 0)$ and the second edge joins $(4,0)$ to $(26, 1)$.
	\end{proof}

	\section{Proof of Theorem \ref{Her}.}
	\begin{proof}[Proof of Theorem \ref{Her}] For proving the result, it is sufficient to show that $H^{\phi}_{2n}(x)$ is irreducible for all integers $n\geq 2$ and $H^{\phi}_{2n+1}(x)$ is $\phi(x)$ times an irreducible polynomial for all non-negative integers $n$ except in the case when $2n+1$ is of the form $3^u$ for some integer $u\geq 2.$  Let $m = 2n$ or $2n+1$ according as $m$ is even or odd. Let $b_i(x)$ with $0\leq i\leq n-1$ belonging to $\Z[x]$  be polynomials having degree less than $\deg\phi(x)$.  Let the content of $(a_nb_0(x))$ is not divisible by any prime less than less than or equal to $m$.  Then, we define
	$$f_1(x) = 
		\frac{a_{n}}{u_{2n}}\phi(x)^{2n}+\sum\limits_{j=0}^{n-1}b_j(x)\frac{\phi(x)^{2j}}{u_{2j}}, ~~\mbox{if $m=2n$, $n\geq 2$  }$$ and 
		$$f_2(x)= \frac{a_{n}}{u_{2n+2}}\phi(x)^{2n}+\sum\limits_{j=0}^{n-1}b_j(x)\frac{\phi(x)^{2j}}{u_{2j+2}}, \mbox{ if } m =2n+1 \mbox{ and } m\neq 3^u \mbox{ for } u\geq 2.$$
		Using Theorems \ref{1.1}, \ref{1.1'}, we see that $f_1(x)$ and $f_2(x)$ are irreducible polynomial over $\Q$.
		
		Observe that 
		$$u_{2j} = 1\cdot 3 \cdots (2j-1) = \frac{(2j)!}{2\cdot 4\cdots 2j} = \frac{(2j)!}{2^j j!}.$$
		Therefore we see that
		$$\frac{H^{\phi}_{2n}(x)}{u_{2n}} = a_n\frac{\phi(x)^{2n}}{u_{2n}} + \sum\limits_{j=0}^{n-1}\binom{n}{j}a_j(x)\frac{\phi(x)^{2j}}{u_{2j}},$$
		and
		$$\frac{H^{\phi}_{2n+1}(x)}{u_{2n+2}\phi(x)} = a_n\frac{\phi(x)^{2n}}{u_{2n+2}} + \sum\limits_{j=0}^{n-1}\binom{n}{j}a_j(x)\frac{\phi(x)^{2j}}{u_{2j+2}}.$$
		So, by taking $b_j(x):= \binom{n}{j}a_j(x)$ in $f_1(x)$ and $f_2(x)$, we have our result.
	\end{proof}
	
 \medskip
  \vspace{-3mm}
 % \newpage

 \end{document}